\documentclass[a4paper,USenglish,cleveref,autoref, nolineno]{socg-lipics-v2019}

\usepackage{amssymb,amsmath}

\usepackage{nameref} 
\usepackage{stmaryrd}
\usepackage[dvipsnames]{xcolor} 
\usepackage[all]{xy} 
\usepackage{tikz}
\usepackage{tikz-cd} 
\usepackage{xparse} 
\usepackage{cancel} 
\usepackage{scrextend} 

\usepackage[ruled, vlined]{algorithm2e} 

\crefname{equation}{equation}{equations}
\crefname{figure}{figure}{figures}
\crefname{equation}{equation}{equations}
\crefname{figure}{figure}{figures}

\def\N{{\mathbb N}}    
\def\Z{{\mathbb Z}}    
\def\R{{\mathbb R}}    

\renewcommand{\phi}{\varphi} 
\newcommand{\libr}{\llbracket} 
\newcommand{\ribr}{\rrbracket} 

\DeclareMathOperator{\Ima}{Im}
\DeclareMathOperator{\Ker}{Ker}
\DeclareMathOperator{\Coker}{Coker}

\DeclareMathOperator{\Span}{Span}

\newcommand{\squarediag}[8]{
    \begin{tikzcd}[ampersand replacement=\&]
        #3 \arrow[r,"#7"] \& #4 \\
        #1 \arrow[r,"#5"'] \arrow[u, "#6"] \& #2 \arrow[u,"#8"']
    \end{tikzcd}
    }
    
    \newcommand{\ses}[6][0]{
        \ifodd#1
        \begin{tikzcd}[ampersand replacement=\&]
            #2 \arrow[r,"#5"] \& #3 \arrow[r,"#6"] \& #4
        \end{tikzcd}
        \else
        \begin{tikzcd}[ampersand replacement=\&]
            0 \arrow[r] \& #2 \arrow[r,"#5"] \& #3 \arrow[r,"#6"] \& #4 \arrow[r] \& 0
        \end{tikzcd}
        \fi
        }
        
\newcommand{\cat}[1]{{\normalfont\mathsf{#1}}} 

\renewcommand{\lim}[1][ ]{\varprojlim_{#1}}

\DeclareMathOperator{\Hom}{Hom}
\DeclareMathOperator{\End}{End}
\DeclareMathOperator{\Ran}{Ran}
\DeclareMathOperator{\Lan}{Lan}

\def\field{\textnormal{\textbf{k}}}

\def\mor{{\rho}}
\def\basemod{M}

\def\abscisse{X}
\def\ordinates{Y}
\def\poset{\abscisse\times\ordinates}

\def\omod{N}

\def\zeromod{\cat{0}}

\def\rec{{R}}

\def\identity{\textnormal{Id}}
\NewDocumentCommand{\indimod}{d[] d<>}{%
  \IfNoValueTF{#1}{\def\rectemp{\rec}}{\def\rectemp{#1}}
  \IfNoValueTF{#2}{\field_{\rectemp}}{\field_{\rectemp,#2}}%
}

\newcommand{\rectangles}{\mathcal{R}}

\newcommand{\rank}{r}

\def\letterfilt{V}
\def\letterdfilt{W}

\NewDocumentCommand{\filt}{d[] m d<>}{%
\IfNoValueTF{#1}{\def\rectemp{\rec}}{\def\rectemp{#1}}
\IfNoValueTF{#3}{\letterfilt_{\rectemp}^{#2}}{\letterfilt_{\rectemp,#3}^{#2}}%
}

\newcommand{\gen}{\gamma}
\newcommand{\rel}{\eta}
\newcommand{\rrel}{\zeta}
\DeclareMathOperator{\grade}{gr}
\DeclareMathOperator{\lub}{lub}

\NewDocumentCommand{\dfilt}{d[] m d<>}{%
\IfNoValueTF{#1}{\def\rectemp{\rec}}{\def\rectemp{#1}}
\IfNoValueTF{#3}{\letterdfilt_{\rectemp}^{#2}}{\letterdfilt_{\rectemp,#3}^{#2}}%
}

\NewDocumentCommand{\filtrate}{d[] d<>}{%
  \IfNoValueTF{#1}{\def\rectemp{\rec}}{\def\rectemp{#1}}
  \IfNoValueTF{#2}{\basemod_{\rectemp}}{\basemod_{\rectemp,#2}}%
}

\def\filt{F}
\def\grid{G}

\newcommand{\dart}[1]{\mathcal{D}_{#1}}

\NewDocumentCommand{\totallysubpos}{o}{\IfNoValueTF{#1}{\textnormal{TSub}}{\textnormal{TSub}(#1)}}

\NewDocumentCommand{\indec}{m d<>}{%
  \IfNoValueTF{#2}{\basemod^{#1}}{\basemod^{#1}_{#2}}%
}
\NewDocumentCommand{\indecgrid}{m d<>}{%
  \IfNoValueTF{#2}{\omod^{#1}}{\omod^{#1}_{#2}}%
}

\newcommand{\op}{\mathrm{op}}

\title{On rectangle-decomposable 2-parameter persistence modules}
\author{Magnus Bakke Botnan}{Vrije Universiteit Amsterdam, Amsterdam, Netherlands}{m.b.botnan@vu.nl}{}{}
\author{Vadim Lebovici}{\'Ecole Normale Sup\'erieure, Paris, France}{vadim.lebovici@ens.fr}{}{}
\author{Steve Oudot}{Inria Saclay, Palaiseau, France}{steve.oudot@inria.fr}{}{}

\authorrunning{M.\,Botnan, V.\,Lebovici and S.\,Oudot}

\Copyright{Magnus Botnan, Vadim Lebovici and Steve Oudot}

\ccsdesc[100]{Mathematics of computing~Algebraic topology }

\keywords{topological data analysis, multiparameter persistence, rank invariant}

\EventEditors{Sergio Cabello and Danny Z. Chen}
\EventNoEds{2}
\EventLongTitle{36th International Symposium on Computational Geometry (SoCG 2020)}
\EventShortTitle{SoCG 2020}
\EventAcronym{SoCG}
\EventYear{2020}
\EventDate{June 23--26, 2020}
\EventLocation{Z\" urich, Switzerland}
\EventLogo{socg-logo}
\SeriesVolume{164}
\ArticleNo{22} 
\usepackage{mathtools}

\newcommand{\mymatrix}[1]{\begin{psmallmatrix} #1 \end{psmallmatrix}}

\begin{document}

\maketitle

\begin{abstract}
  This paper addresses two questions: (a) can we identify a sensible
  class of 2-parameter persistence modules on which the rank invariant
  is complete? (b) can we determine efficiently whether a given
  2-parameter persistence module belongs to this class?  We provide
  positive answers to both questions, and our class of interest is
  that of rectangle-decomposable modules. Our contributions include:
  on the one hand, a proof that the rank invariant is complete on
  rectangle-decomposable modules, together with an inclusion-exclusion
  formula for counting the multiplicities of the summands; on the
  other hand, algorithms to check whether a module induced in homology
  by a bifiltration is rectangle-decomposable, and to decompose it in
  the affirmative, with a better complexity than state-of-the-art
  decomposition methods for general 2-parameter persistence
  modules. Our algorithms are backed up by a new structure theorem,
  whereby a 2-parameter persistence module is rectangle-decomposable
  if, and only if, its restrictions to squares are. This local
  characterization is key to the efficiency of our algorithms, and it
  generalizes previous conditions derived for the smaller class of
  block-decomposable modules. It also admits an algebraic formulation
  that turns out to be a weaker version of the one for
  block-decomposability. By contrast, we show that general
  interval-decomposability does not admit such a local
  characterization, even when locality is understood in a broad
  sense. Our analysis focuses on the case of modules indexed over
  finite grids, the more general cases are left as future work.
\end{abstract}

\section{Introduction}
A \emph{persistence module} $M$ over a subset $U\subseteq \R^d$ is a collection of vector spaces $\{M_t\}_{t\in U}$ and linear maps $\mor_s^t:=M(s\leq t)\colon M_s\to M_t$ with the property that $\mor_s^s$ is the identity map and $\mor_t^u \circ \mor_s^t = \mor_s^u$ for all $s\leq t\leq u\in U$. Here $s\leq t$ if and only if $s_i\leq t_i$ for all $i\in \{1, 2, \ldots, d\}$. In the language of category theory, a persistence module $M$ is a functor $M\colon U\to \cat{vec}$ where $\cat{vec}$ is the category of vector spaces and the partially ordered set $U$ is considered as a category in the obvious way. In this setting, {\em morphisms} between persistence modules are natural transformations $M\Rightarrow N$ between functors, defined by collections of linear maps $\{\phi_t: M_t\to N_t\}_{t\in U}$ such that $\phi_t\circ M(s\leq t) = N(s\leq t) \circ \phi_s$ for all $s\leq t\in U$. Their kernels, images and cokernels, as well as products, direct sums and quotients of persistence modules, are defined pointwise at each index $t\in U$. Similarly, an \emph{isomorphism} between two persistence modules is a natural isomorphism between them. We will refer to the case $d=1$ as \emph{single-parameter persistence}, and for $d\geq 2$ we will use the term \emph{multi-parameter persistence}.
\begin{remark}
Throughout this paper we will work exclusively with finite-dimensional vector spaces over a fixed field $\field$. When finite-dimensionality is emphasized we will refer to the persistence module as being pointwise finite-dimensional (pfd). 
\end{remark}

Single-parameter persistence modules are typically obtained through the application of homology to a filtered topological space. This process is known as \emph{persistent homology} and has found a wide range of applications to the sciences, as well as to other parts of mathematics such as symplectic geometry. See \cite{edelsbrunner2008persistent, oudot2015persistence} for an introduction to persistent homology. What makes such persistence modules particularly amenable to data analysis is that they can be completely described by multisets of intervals in $\R$ called \emph{barcodes}~\cite{Crawley-Boevey2012}. Such a collection of intervals can then in turn be used to extract topological information from the data at hand, and further utilized in statistics and machine learning. 
We now give an example of this structure theorem in the simple case of $U = \{1,2,3\}\subseteq \R$.

\begin{example}\label{ex:1D}
Consider the following sequence of vector spaces and linear maps
\[ \field^2 \xrightarrow{\begin{bmatrix} 1 & 1 \\ 0 & 1 \end{bmatrix}} \field^2 \xrightarrow{[1~-1]} \field.\]
By replacing the basis $\{e_1, e_2\}$ of the middle vector space $\field^2$ with the basis $\{e_1, e_1+e_2\}$ we get the following matrix representations of the linear maps 
\[ \field^2 \xrightarrow{\begin{bmatrix} 1 & 0 \\ 0 & 1 \end{bmatrix}} \field^2 \xrightarrow{[1~0]} \field = \left( \field\xrightarrow{1} \field\xrightarrow{1} \field\right) \oplus \left (\field\xrightarrow{1} \field \to 0\right).\]
The two persistence modules on the right-hand side are uniquely specified by their supports $\{1,2,3\}$ and $\{1,2\}$, respectively. Their supports give rise to the barcode which in this case is given by $\{\{1,2,3\}, \{1,2\}\}$. 
\end{example}

As illustrated by \Cref{ex:1D}, a persistence module can be recovered from its barcode thanks to the notion of \emph{indicator modules}: for $\poset \subseteq \R^2$ and a subset $Q\subseteq\poset$, the indicator module of $Q$, denoted $\indimod[Q]$, is defined by
\begin{align*}
   \field_{Q,t}=
    \begin{cases}
        \field &(t\in Q) \\
        0 &(t\notin Q)
    \end{cases} & & 
      \indimod[Q](s\leq t) =
    \begin{cases}
        \identity_\field &\mbox{ if } s \mbox{ and } t \in Q, \\
        0 &\mbox{ else}.
    \end{cases}
\end{align*}
By convention, we set $\indimod[\emptyset]= \zeromod$. A persistence module is an \emph{interval module} if it is the indicator module of an interval\footnote{In the poset $X\times Y$, we say that $Q$ is an interval if it is convex and zigzag path-connected, i.e if between any two points $p,q\in Q$, there exists a zigzag path $p\leq p_1 \geq p_2\leq \dots \geq p_n \leq q$ with $p_i$'s in $Q$.}.
Note that, just as choosing a basis for a vector space is not canonical, there may be many ways of decomposing a single-parameter persistence module into a direct sum of such interval modules. However, just as for the dimension of a finite-dimensional vector space, the associated barcode given by the multiset of interval supports of the summands is independent of the chosen decomposition~\cite{Azumaya1950}.

Another desirable property of single-parameter persistence modules~$M$ is that they are completely described up to isomorphism by the \emph{rank invariant}, i.e. the collection of ranks~$\rank(s,t) = \textnormal{rank}(\basemod(s\leq t))$  for all $s\leq t$. This can easily be verified in the previous example, and more generally, for any pfd persistence module~$M$ indexed over a finite set~$\libr 1, n\ribr$, the following inclusion-exclusion formula (also known as the {\em persistence measure}~\cite{chazal2016structure,cohen2007stability}) gives the multiplicity $m(s,t)$ of any interval $\libr s, t\ribr$ in the barcode of~$M$: 
\begin{equation}\label{eq:incl_excl_1d}
m(s,t) = \rank(s,t) - \rank(s-1,t) - \rank(s,t+1) + \rank(s-1,t+1).
\end{equation}

Many applications do however naturally come equipped with multiple parameters, and for such applications it is natural to consider multi-parameter persistence modules, see e.g. the introduction of \cite{bauer2019cotorsion} for an example of how multi-parameter persistence connects to hierarchical clustering. Let us first consider the simplest instantiation of 2-parameter persistence modules, namely modules indexed by the square $S = \{a=(0,0), b=(1,0), c=(0,1), d=(1,1)\}\subseteq \R^2$. 
\begin{example}\label{ex:square}
The persistence module on the left-hand side below can be transformed into the one on the right-hand side via a change of basis at the vertices: 
\[  \begin{tikzcd}[arrows=-stealth, ampersand replacement=\&, column sep=4em, row sep=3em]
\field^2\rar{\mymatrix{1 & -1\\0 & 1}} \& \field^2 \& \& \field^2\rar{\mymatrix{1 & 0 \\0 & 1}} \& \field^2 \\
\field\rar{1}\uar{\mymatrix{1\\1}} \& \field \uar[swap]{\mymatrix{0\\1}}  \& \& \field\rar{1}\uar{\mymatrix{0\\1}} \& \field \uar[swap]{\mymatrix{0\\1}}
\end{tikzcd} \]
In turn, the persistence module on the right-hand side  is the direct sum
\[  \begin{tikzcd}[arrows=-stealth, ampersand replacement=\&, column sep=4em, row sep=3em]
\field \rar{1} \& \field  \\
0 \rar{0} \uar{0} \& 0 \uar[swap]{0} \end{tikzcd}
\quad \oplus \quad
\begin{tikzcd}[arrows=-stealth, ampersand replacement=\&, column sep=4em, row sep=3em]
\field \rar{1} \& \field  \\
\field \rar{1} \uar{1} \& \field \uar[swap]{1}\end{tikzcd}\]
Just as in \cref{ex:1D}, these persistence modules are completely defined by their support.  We define the barcode of the aforementioned persistence module to be the (multi-)set of supports of its summands, namely
$\{\{c,d\}, \{a,b,c,d\}\}$.
\end{example}
Although commutative diagrams like the one in the previous example may appear unwieldy at first glance, such persistence modules can --- just as in the single-parameter case --- be completely described (up to isomorphism) by a multiset of elements from
\begin{equation}
  \label{eq:def-int-square}
    \mathcal{I}:=\{\{a\},\{b\}, \{c\}, \{d\},\{a,b\},\{a,c\},\{b,d\},\{c,d\},\{a,b,c\},\{b,c,d\},\{a,b,c,d\}\},
\end{equation}
called \emph{intervals} in the $2\times 2$~grid. See e.g. Figure 13 in \cite{escolar2016persistence}.
However, in contrast to the single-parameter case, the rank invariant on  persistence modules indexed by $S$ is no longer a {\em complete} invariant, i.e. it does not fully determine the  isomorphism type of such modules. For instance, two persistence modules with barcodes $\{\{a,b,c\}, \{a\}\}$ and $\{\{a,b\}, \{\{a,c\}\}$ are non-isomorphic yet exhibit the same rank invariant. 
\begin{example}
  \label{ex:3-2-indecomposable}
Consider the following two persistence modules:
%
\[  \begin{tikzcd}[arrows=-stealth, ampersand replacement=\&, column sep=4em, row sep=3em]
\field\rar{\mymatrix{1\\0}} \& \field^2\rar{\mymatrix{1& 0}} \& \field \& \& \field\rar{\mymatrix{1\\1}} \& \field^2\rar{\mymatrix{1& 0}} \& \field
\\
0\rar{0}\uar{0} \& \field \uar{\mymatrix{1\\0}}\rar{1} \& \field \uar{1}
\& \& 0\rar{0}\uar{0} \& \field \uar{\mymatrix{1\\0}}\rar{1} \& \field \uar{1}
\end{tikzcd} \]
%
%
The diagram to the left can easily be seen to be composed of two interval summands in the $3\times 2$~grid. By contrast, the diagram to the right is indecomposable: there exists no change of basis for which this persistence module can be written as a direct sum of persistence modules in a non-trivial way. Again, the two modules have the same rank invariant.
\end{example}

In the setting of no more than four columns and two rows, results from the field of representation theory of quivers show that there exists a finite set of building blocks (indecomposable modules) from which every persistence module can be built (via direct sums, and up to isomorphism). Based on this, one can associate a well-defined barcode-like structure to such a module by counting the multiplicity of every summand in the decomposition. The inclusion of such grids into topological data analysis was inspired by a problem in materials science \cite{escolar2016persistence}. For five or more columns the theory becomes increasingly complex. In particular, for six or more columns there is no way to parametrize a set of building blocks in any reasonable way\footnote{The underlying graph, called a quiver, is known to be of  \emph{wild} representation type.}. This is  a major obstacle to the development of the theory of multi-parameter persistence.

A natural question to consider then is whether one can endow multi-parameter persistence modules with additional structure in order to enforce nice decomposition theorems akin to that of single-parameter persistence. One such setting coming from computational topology was identified in \cite{bendich2013homology,carlsson2009zigzag}, and further generalized in \cite{Cochoy2016}, where it is shown that the so-called \emph{strongly exact} 2-parameter persistence modules indexed over $\R^2$ are determined (up to isomorphism) by a multiset of particularly simple planar rectangular regions called {\em blocks}. Basically, a block is either an upper-right or lower-left quadrant, or a horizontal or vertical infinite band. The great advantage of this condition is that it can be checked locally: a 2-parameter persistence module (called a {\em bimodule} for short) is block-decomposable if, and only if, its restriction to any square as in Example~\ref{ex:square} is block-decomposable.  

\medskip

\noindent{\bf Contributions.} In this paper we address two important follow-up questions:
\begin{itemize}
\item Can we work out conditions such as above for larger classes of bimodules?
\item Can we identify classes of bimodules for which the rank invariant is complete?
\end{itemize}
Our answers to both questions are positive, and the two classes of bimodules  turn out to be the same, namely that of {\em rectangle-decomposable} bimodules, which by definition are determined (up to isomorphism) by a multiset of \emph{rectangles}, i.e subsets $R$ of the form $R = I \times J \subseteq \R^2$ where $I$ and $J$ are intervals in $\R$. Specifically, a bimodule is rectangle-decomposable if it decomposes into a direct sum of  \emph{rectangle modules}, i.e. indicator modules of rectangles.

Our local condition for rectangle decomposability, called \emph{weak exactness}, is a weaker version of the condition for block decomposability, in that it allows all types of rectangular shapes in the local squares' decompositions, as opposed to just blocks. More precisely, calling $\mathcal{R}$ the following subset of $\mathcal{I}$:
\begin{equation}
  \label{eq:def-rec-square}
\mathcal{R}=\{\{a\},\{b\},\{c\},\{d\},\{a,b\},\{a,c\},\{b,d\},\{c,d\},\{a,b,c,d\}\},
\end{equation}
\begin{definition}[Weak exactness]\label{def:weak-exactness-geom}
Given subsets $X,Y$ of~$\R$, a persistence module $M\colon X\times Y\subseteq \R\times \R \to \cat{vec}$ is \emph{weakly exact}
if the barcode of the following square
\begin{equation}\label{eq:weak-exactness}
\squarediag{\basemod_s}{\basemod_{(t_x,s_y)}}{\basemod_{(s_x,t_y)}}{\basemod_t}{\mor_s^{(t_x,s_y)}}{\mor_s^{(s_x,t_y)}}{\mor_{(s_x,t_y)}^t}{\mor_{(t_x,s_y)}^t}
\end{equation}
consists of elements from $\mathcal{R}$ for all indices $s\leq t$ in $X\times Y$. 
\end{definition}
By comparison, the strong exactness condition replaces $\mathcal{R}$ by $\mathcal{B} = \mathcal{R} \setminus \{\{b\}, \{c\}\}$.

\begin{example}
The persistence module to the left below is strongly exact, while the one to the right is only weakly exact, and the persistence modules in \Cref{ex:3-2-indecomposable} are not even weakly exact (each time the weak or strong exactness condition fails on the outermost rectangle): 
\[  \begin{tikzcd}[arrows=-stealth, ampersand replacement=\&, column sep=4em, row sep=3em]
\field\rar{\mymatrix{1\\0}} \& \field^2\rar{\mymatrix{1& 0\\0& 1}} \& \field^2
\& \&
\field\rar{\mymatrix{1\\0}} \& \field^2\rar{\mymatrix{0& 0\\0& 1}} \& \field^2
\\
0\rar{0}\uar{0} \& \field \uar{\mymatrix{0\\1}}\rar{1} \& \field \uar{\mymatrix{0\\1}}
\& \&
0\rar{0}\uar{0} \& \field \uar{\mymatrix{0\\1}}\rar{1} \& \field \uar{\mymatrix{0\\1}}
\end{tikzcd} \]
\end{example}

Our analysis focuses on the case of modules indexed over finite grids\footnote{A \emph{finite grid} is the product of two finite subsets of $\R$. Note that any finite grid can be identified with a grid of the form $\libr 1, n\ribr \times \libr 1, m\ribr$ for appropriate choices of $n,m\in\N$.}, the more general cases are left as future work. Our contributions summarize as follows:
\begin{itemize}
\item In Section~\ref{sec:rank_invariant_completeness} we prove that the rank invariant is complete on the class of rectangle-decomposable bimodules (Theorem~\ref{thm:rank_invariant_complete}). To this end, we generalize the inclusion-exclusion formula~\eqref{eq:incl_excl_1d} to our setting.
Note that our result also follows indirectly from an inclusion-exclusion formula for a generalization of the rank invariant for interval-decomposable modules~\cite[Prop. 7.13]{kim2018generalized}, but that we provide an explicit statement together with a simple and direct proof.
\item In Section~\ref{sec:rank_invariant_computation} we show that the rank invariant of a simplicial bifiltration with a total of $n$ simplices can be computed in $O(n^4)$ time (Theorem~\ref{thm:rank_invariant_computation}). This result in itself is not new, however, combined with our inclusion-exclusion formula, it yields an $O(n^4)$ time algorithm for computing the barcode of a persistence bimodule that is known to be rectangle-decomposable (Corollary~\ref{cor:rectangle_decomposition_computation}). This is an improvement over merely applying some state-of-the-art algorithm for computing decompositions of general 2-parameter persistence modules, which would take $O(n^{2\omega+1})$ time where $2\leq \omega < 2.373$ is the exponent for matrix multiplication~\cite{dey2019generalized}.
\item In Section~\ref{sec:local_condition} we propose an algebraic formulation   of our weak exactness condition (\Cref{def:weak-exactness-alg}). This formulation  turns out to be equivalent to \Cref{def:weak-exactness-geom} and to global rectangle-decomposability (the central mathematical result in the paper), specifically:
\begin{theorem}
\label{main:rec-dec}
Let $M$ be a pfd persistence module indexed over $X\times Y$, where $X,Y$ are finite subsets of $\R$. Then, $M$ is rectangle-decomposable if and only if $M$ is weakly exact. 
\end{theorem}
\item In Section~\ref{sec:check} we leverage this result to derive an $O(n^{2+\omega})$-time algorithm for checking the rectangle-decomposability of persistence bimodules induced in homology from simplicial bifiltrations with at most $n$ simplices (Theorem~\ref{thm:check_weak_exactness}). Once again, this is an improvement over applying some state-of-the-art algorithm for computing decompositions of general 2-parameter persistence modules and then checking the summands one by one.
\item In Section~\ref{sec:counter-examples} we investigate the existence of similar local characterizations for larger classes of interval-decomposable modules. First, we show that interval-decomposability itself cannot be characterized locally, even when testing on arbitrary strict subgrids and not just squares.  Second, we show that decomposability with respect to certain classes of indecomposables in-between the rectangle modules and the interval modules cannot be characterized locally either when testing on squares.
\item Finally, in Section~\ref{sec:example} we show how rectangle-decomposable modules arise from (sufficiently tame) real-valued functions on a topological space. This is then used to give a new proof of the \emph{pyramid basis theorem} of \cite{bendich2013homology}.
\end{itemize}
%

\section{Completeness of the rank invariant}
\label{sec:rank_invariant_completeness}

Suppose in this section that $X, Y$ are subsets of~$\Z$.
\begin{theorem}\label{thm:rank_invariant_complete}
  The isomorphism type of any pfd rectangle-decomposable persistence module~$\basemod$ over~$X\times Y$ is fully determined by the rank invariant of~$\basemod$.
\end{theorem}
The proof consists in showing that the multiplicity~$m(s,t)$ of each individual rectangle module~$\field_{\libr s_x, t_x\ribr\times\libr s_y, t_y\ribr}$  in the decomposition of~$\basemod$ is given by the inclusion-exclusion formula~\eqref{eq:incl-excl_rank} below, which involves only the rank invariant~$\rank: (X\times Y)^2\to\N$ of~$\basemod$ --- with the convention that $\rank(s,t)=0$ whenever $s\nleq t$. This formula is the analogue, in the category of rectangle-decomposable pfd bimodules, of the inclusion-exclusion formula~\eqref{eq:incl_excl_1d} for counting the multiplicities of interval summands in one-parameter persistence.

Fix arbitrary indices $s\leq t\in X\times Y$. Recall that the rank of~$(A\oplus B)(s\leq t)$ is equal to the sum of the ranks of~$A(s\leq t)$ and~$B(s\leq t)$. Meanwhile, for any summand~$\field_R$ of~$\basemod$, the rank of~$\field_R(s\leq t)$ is~$1$ if $s,t\in R$ and $0$ otherwise. Therefore, $\rank(s, t)$ counts (with multiplicity) the number of summands of~$\basemod$ whose rectangle support  contains both~$s$ and~$t$. Then, denoting by $m(s, t^+)$ the number of (rectangle) summands whose support contains~$t$ and has~$s$ as lower-left corner, we have the folllowing inclusion-exclusion formula:
\begin{equation}\label{incl-excl_m+}
m(s, t^+) = \rank(s, t) - \rank((s_x-1, s_y), t) - \rank((s_x, s_y-1), t) + \rank((s_x-1, s_y-1), t).
\end{equation}
This formula can be interpreted as follows: a rectangle containing~$t$ has $s$ as lower-left corner if and only if it contains $s$ but neither $(s_x-1, s_y)$ nor $(s_x, s_y-1)$; and it contains both $(s_x-1, s_y)$ and $(s_x, s_y-1)$ if and only if it contains $(s_x-1, s_y-1)$.

Using the same approach at~$t$, we can now compute the number~$m(s,t)$ of summands of~$\basemod$ whose support has $s$ as lower-left corner and~$t$ as upper-right corner (i.e. is the rectangle~$\libr s_x, t_x\ribr\times\libr s_y, t_y\ribr$). The corresponding inclusion-exclusion formula is:
\begin{equation}\label{eq:incl-excl_m}
m(s, t) = m(s, t^+) - m(s, (t_x+1, t_y)^+) - m(s, (t_x, t_y+1)^+) + m(s, (t_x+1, t_y+1)^+).
\end{equation}
Combining~\eqref{incl-excl_m+} and~\eqref{eq:incl-excl_m} together gives the  desired inclusion-exclusion formula for the multiplicity~$m(s,t)$ of the summand~$\field_{\libr s_x, t_x\ribr\times\libr s_y, t_y\ribr}$ in the decomposition of~$\basemod$ from the rank invariant, hence completing the proof of Theorem~\ref{thm:rank_invariant_complete}, namely:  
\begin{equation}\label{eq:incl-excl_rank}
  \begin{gathered}
    \begin{array}{rcl}
      m(s, t) & = & \rank(s, t) - \rank((s_x-1, s_y), t) \\[0.5em]
      &&\hspace{10pt} - \rank((s_x, s_y-1), t) + \rank((s_x-1, s_y-1), t) \\[0.5em]
      && - \rank(s, (t_x+1, t_y)) + \rank((s_x-1, s_y), (t_x+1, t_y))\\[0.5em]
      &&\hspace{10pt} + \rank((s_x, s_y-1), (t_x+1, t_y)) - \rank((s_x-1, s_y-1), (t_x+1, t_y)) \\[0.5em]
      && - \rank(s, (t_x, t_y+1)) + \rank((s_x-1, s_y), (t_x, t_y+1))\\[0.5em]
      &&\hspace{10pt} + \rank((s_x, s_y-1), (t_x, t_y+1)) - \rank((s_x-1, s_y-1), (t_x, t_y+1)) \\[0.5em]
      && + \rank(s, (t_x+1, t_y+1)) - \rank((s_x-1, s_y), (t_x+1, t_y+1))\\[0.5em]
      &&\hspace{10pt} - \rank((s_x, s_y-1), (t_x+1, t_y+1)) + \rank((s_x-1, s_y-1), (t_x+1, t_y+1)).
    \end{array}
  \end{gathered}
\end{equation}

\section{Computing the rank invariant and rectangle decompositions}
\label{sec:rank_invariant_computation}
Let $\filt$ be a simplicial bifiltration with $n$ simplices in total. Assume without loss of generality that~$\filt$ is indexed over the grid $\grid=\libr 1, n\ribr \times \libr 1, n\ribr$, for any larger indexing grid must contain arrows with identity maps that can be pre- or post-composed, and any smaller grid can be enlarged by inserting arrows with identity maps. Assume further that each arrow $\filt_{(i,j)} \to \filt_{(i+1,j)}$ or $\filt_{(i,j)} \to \filt_{(i,j+1)}$ is either an identity map or the insertion of a single simplex. 
We also fix a homology degree~$p$.
\begin{theorem}\label{thm:rank_invariant_computation}
  Given the above input, the rank invariant of the persistence bimodule~$M$ induced by~$\filt$ in $p$-th homology can be computed in $O(n^4)$~time.
\end{theorem}
A proof of this result can be found in D. Morozov's Ph.D. thesis~\cite[Section~4.4.2]{morozov2008homological}. We reproduce it below for completeness, with a slight adaptation that allows us to avoid assuming that $\filt$ is 1-critical\footnote{Let us also point out that the algorithm given in the conference version of this paper~\cite{botnan_et_al:LIPIcs:2020:12180} was incorrect.}. Before giving the proof, let us mention that this theorem, combined with the inclusion-exclusion formula~\eqref{eq:incl-excl_rank}, gives an $O(n^4)$-time algorithm to compute the barcode of~$\filt$ assuming that~$M$ is rectangle-decomposable: once the rank invariant of~$M$ has been computed, iterate over all the pairs~$(s,t)$ with $s\leq t\in \grid$ and, for each one of them, apply the formula in constant time to get the multiplicity of the rectangle module~$\field_{\libr s_x, t_x \ribr\times\libr s_y,t_y\ribr}$  in the decomposition of~$M$. Thus,
\begin{corollary}\label{cor:rectangle_decomposition_computation}
  Computing the decomposition of a rectangle-decomposable module over~$X\times Y$  induced in homology by a bifiltration with $n$ simplices in total can be done in $O(n^4)$~time.
\end{corollary}
This complexity compares favorably to that of the currently best known algorithm for computing direct-sum decompositions of general persistence bimodules\footnote{Let us also point out that our approach does not suffer from the limitation of the algorithm of~\cite{dey2019generalized}, which is that no two generators or relations in a minimal presentation of~$M$ can  have the same grade.}, which is~$O(n^{2\omega+1})$ where $2\leq \omega < 2.373$ is the exponent for matrix multiplication~\cite{dey2019generalized}.

\begin{figure}[t]
\centering
\includegraphics[width=\textwidth]{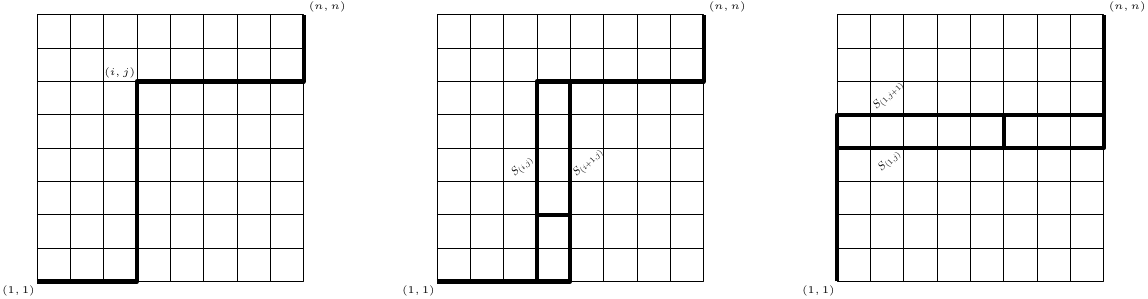}
\caption{Left: the stair~$S_{(i,j)}$. Center: transitioning from $S_{(i,j)}$ to $S_{(i+1,j)}$ via a sequence of $O(n)$ intermediate paths with 2 steps each. Right: transitioning from $S_{(1,j)}$ to $S_{(1,j+1)}$.}
\label{fig:stairs}
\end{figure}

Let us now provide the algorithm for Theorem~\ref{thm:rank_invariant_computation}. First, we provide a simplified algorithm that runs in  $O(n^{2+\omega})$~time. Consider all the paths of the form $(1,1) \to \cdots \to (i,1) \to \cdots \to (i,j) \to \cdots \to (n,j) \to \cdots \to (n,n)$ in the $n\times n$ grid, where $(i,j)\in\libr 1,n\ribr^2$ is arbitrary, as illustrated in Figure~\ref{fig:stairs}~(left). We call such a path a {\em stair}, denoted by $S_{(i,j)}$, and we call the corresponding index~$(i,j)$ its {\em nosing}. Note that all the stairs whose nosing is of the form $(i,1)$ or $(n,j)$ are in fact identical to the path $(1,1)\to \cdots \to (n,1) \to \cdots \to (n,n)$, while all the other stairs with different nosings are different.  The key observation is that, for any pair of comparable indices $(i,j)\leq (i',j')$, the stair with nosing $(i,j')$ passes through both indices. Computing the rank invariant can then be done by iterating over all the stairs, for instance in lexicographical order of the coordinates of their respective nosings, and for each such stair~$S_{(i,j)}$, by computing the persistence barcode of the 1-parameter restriction~$F|_{S_{(i,j)}}$ and then using this barcode to report the ranks between all the grid indices encountered along the path. This takes $O(n^\omega)$  per stair, using the 1-parameter persistence algorithm based on fast matrix multiplication~\cite{milosavljevic2011zigzag}, and as there are $O(n^2)$ stairs in total, the overall running time of the algorithm is $O(n^{2+\omega})$.

In order to reduce the overall complexity to~$O(n^4)$, we exploit the additional observation that, to transition between two consecutive stairs in lexicographical order of the coordinates of their nosings, say $S_{(i,j)}$ and $S_{(i+1,j)}$, one can go through a sequence of $O(n)$ intermediate paths of the form $(1,1) \to \cdots \to (i,1) \to \cdots \to (i,k) \to (i+1,k) \to \cdots \to (i+1, j) \to \cdots \to (n,j) \to \cdots \to (n,n)$, where $k$ ranges from~$2$ to~$j-1$, as illustrated in Figure~\ref{fig:stairs}~(center). Any two consecutive paths in this sequence differ only at a single cell of the grid $\libr 1,n\ribr^2$, therefore the restrictions of~$F$ to these two paths either do not differ, or differ by one simplex being inserted one step earlier or later, or by two consecutive simplex insertions being exchanged. In any situation, the persistence barcode can be updated in $O(n)$~time using the vineyards algorithm~\cite{cohen2006vines}. The update of the barcode from $S_{(i,j)}$ to $S_{(i+1,j)}$ then takes $O(n^2)$ time. Likewise, we can compute the barcode of $F|_{S_{(1,j+1)}}$ by transitioning from $S_{(1,j)}$ via an intermediate sequence of $O(n)$~paths differing at a single cell in the grid each time, as illustrated in Figure~\ref{fig:stairs}~(right). Thus, the barcode of $F|_{S_{(1,j+1)}}$ is also obtained in $O(n^2)$~time, and so the overall running time of the algorithm is reduced to~$O(n^4)$. This concludes the proof of Theorem~\ref{thm:rank_invariant_computation}.

\section{Algebraic formulation of weak exactness}
\label{sec:local_condition}
As shown in~\cite{botnan2018decomposition,Cochoy2016}, a persistence module $M\colon X\times Y\subseteq \R\times \R \to \cat{vec}$ is strongly exact if, and only if, the following sequence induced by~\eqref{eq:weak-exactness} is exact for all indices $s\leq t \in X\times Y$:
\begin{equation}\label{eq:exact-seq}
\xymatrix{M_s \ar^-{(\mor_s^{(t_x, s_y)}, \rho_s^{(s_x, t_y)})}[rrr] &&& M_{(t_x, s_y)} \oplus M_{(s_x, t_y)} \ar^-{\mor_{(t_x, s_y)}^t - \mor_{(s_x, t_y)}^t}[rrr] &&& M_t.
}\end{equation}
Similarly, we can characterize weak exactness (\Cref{def:weak-exactness-geom}) algebraically:
\begin{definition}[Algebraic weak exactness]\label{def:weak-exactness-alg}
A persistence module $M\colon X\times Y\subseteq \R\times \R \to \cat{vec}$  is called \emph{algebraically weakly exact} if  the following equalities hold for all $s\leq t \in X\times Y$:
\begin{align*}
    \Ima\mor_s^t &= \Ima\mor_{(t_x,s_y)}^t \cap \Ima\mor_{(s_x,t_y)}^t, \\
    \Ker\mor_s^t &= \Ker\mor_s^{(t_x,s_y)} + \Ker\mor_s^{(s_x,t_y)}.
\end{align*}
\end{definition}
This condition holds in particular when the sequence~\eqref{eq:exact-seq} is exact, but not only. Indeed, as can be checked easily,
any rectangle (not just block) module is algebraically weakly exact. So is any rectangle-decomposable pfd persistence bimodule, since the property is obviously preserved under taking direct sums of pfd persistence bimodules. The converse holds as well:
\begin{theorem}[Decomposition of algebraically weakly exact pfd bimodules]
\label{thm:weak-exact-rec-dec}
For any algebraically weakly exact pfd persistence module $\basemod$ over a finite grid $(\poset,\leq)$, there is a unique multiset $\rectangles{\basemod}$ of rectangles of $\poset$ such that:
\begin{equation*}
    \basemod \cong \bigoplus_{R \in \rectangles{\basemod}} \field_R.
\end{equation*}
\end{theorem}
Since this result holds in particular for persistence bimodules indexed over squares, it ensures that a pfd persistence module over a square is algebraically weakly exact if, and only if, it is rectangle-decomposable. Hence the equivalence between weak exactness (Definition~\ref{def:weak-exactness-geom}) and algebraic weak exactness (Definition~\ref{def:weak-exactness-alg}), and the correctness of~\Cref{main:rec-dec}.

\medskip

The rest of this section is devoted to the proof of Theorem~\ref{thm:weak-exact-rec-dec}. From this point on, and until the end of the section, whenever we talk about {\em weak exactness} we refer consistently to the algebraic formulation from Definition~\ref{def:weak-exactness-alg}.

\subsection{A preliminary remark concerning submodules and summands}
A morphism $f : \basemod\to\omod$ between two persistence modules over $(\poset,\leq)$ is a \emph{monomorphism} (resp. \emph{epimorphism}) if for every $t\in\poset$, $f_t : \basemod_t \to \omod_t$ is injective (resp. surjective). We say that a monomorphism $f:\basemod \to \omod$ between two persistence modules $\basemod$ and $\omod$ \emph{splits} if there is a morphism $g : \omod\to\basemod$ such that $g\circ f = \identity_\basemod$. If every monomorphism with domain $\basemod$ splits, we say that $\basemod$ is an {\em injective persistence module}. 

It is not true that any submodule of a persistence module is a summand. However, if $f:\basemod \to \omod$ is a monomorphism between two persistence modules $\basemod$ and $\omod$ which splits, it is well known that there is a direct sum decomposition $\omod \cong \basemod \oplus \Coker(f)$. Therefore, an injective submodule of a persistence module is a summand thereof. 
In our analysis we will often use the following result:
\begin{lemma}
  \label{lem:dquad-injective}
  For any indices $k\in\libr 1,n\ribr$ and $l\in\libr 1,m\ribr$, the indicator module $\indimod[\libr 1,k\ribr \times \libr 1,l\ribr]$ is an injective persistence module over $\libr 1, n\ribr \times \libr 1, m\ribr$.
\end{lemma}
\begin{proof}
    This lemma is a consequence of \cite[Lem.~2.1]{botnan2018decomposition} since the subset $\libr 1,k\ribr \times \libr 1,l\ribr$ is clearly a directed ideal of the poset $\libr 1, n\ribr \times \libr 1, m\ribr$, following the definition of \cite[Sec.~2.1]{botnan2018decomposition}.
  \end{proof}
  
\subsection{Proof of Theorem~\ref{thm:weak-exact-rec-dec}}

Uniqueness of the decomposition follows directly from Krull-Schmidt-Remak-Azumaya's theorem \cite{Azumaya1950}, since the endomorphism ring of any rectangle module is clearly isomorphic to~$\field$ and thus local. We therefore focus on  the existence of a decomposition into rectangle summands. Our proof proceeds by induction on the poset of grid dimensions $(n,m)$, also viewed as a subposet of~$\R^2$ equipped with the product order:

\bigskip

\noindent $\bullet$ Our base cases are when $n=1$ or $m=1$. The result  is then a direct consequence of Gabriel's theorem \cite{Gabriel1972}, which asserts that $\basemod$ decomposes as a direct sum of interval modules, each interval being a rectangle of width~$1$.

\bigskip

\noindent $\bullet$ Fix $n > 1$ and $m > 1$, and assume that the
result is true for all grids of sizes $n'\times m'$ such that $(n',
m') < (n,m)$. Fix a persistence module $\basemod$ over $\libr 1, n\ribr \times \libr 1, m\ribr$ that is pfd and
weakly exact. Observe that $\basemod$ has finite \emph{total dimension} $\sum_{t\in\libr 1, n\ribr \times \libr 1, m\ribr}\dim\basemod_t$, so we know
from a simple induction that~$\basemod$ decomposes as a direct sum of
indecomposables---see \cite[Theorem 1.1]{botnan2018decomposition} for a more general statement. As any summand of a weakly exact module is again weakly exact, we may restrict our attention to pfd indecomposable modules. For the sake of contradiction, assume that $\basemod$ is pfd, weakly exact, indecomposable, and not isomorphic to a rectangle module. Then:

\begin{claim}\label{lem:zero-map}
The map $\mor_{(1,1)}^{(n,m)}$ is zero.
\end{claim}
\begin{proof}
Suppose the contrary. Then we have $\Ker \mor_{(1,1)}^{(n,m)} \subsetneq \basemod_{(1,1)}$. Let $\alpha\in \basemod_{(1,1)}\setminus \Ker\mor_{(1,1)}^{(n,m)}$. The submodule $\omod$ of $\basemod$ spanned by the collection of images $(\mor_{(1,1)}^{(i,j)}(\alpha))_{(i,j)\in \llbracket 1,n\rrbracket \times \llbracket 1, m\rrbracket}$ is isomorphic to~$\field_{\llbracket 1, n\rrbracket \times \llbracket 1, m\rrbracket}$, an injective persistence module by \Cref{lem:dquad-injective}. Hence, $\omod$ is a summand of~$\basemod$, contradicting that $\basemod$ is not isomorphic to a rectangle module.
\end{proof}
\begin{claim}\label{lem:injective_mapping}
The space $\basemod_{(1,1)}$ maps injectively to the nodes of the grid $\llbracket 1, n-1\rrbracket \times \llbracket 1, m-1\rrbracket$.
\end{claim}
\begin{proof}
  Let us restrict~$\basemod$ to the grid $\llbracket 1, n-1\rrbracket \times \llbracket 1, m\rrbracket$. The restriction --- denoted by $\omod$---may no longer be indecomposable, however it is still pfd and weakly exact, therefore our induction hypothesis asserts that~$\omod$ decomposes as a finite (internal) direct sum  where each summand is isomorphic to some rectangle module. Consider any one of these summands, say~$\omod'\cong\field_{\rec'}$, such that $(1,1)\in \rec'$. Then, we claim that $(n-1,1)\in \rec'$ as well. Indeed, otherwise, one can extend $\omod'$ to a persistence module over~$\libr 1, n\ribr \times \libr 1, m\ribr$ by putting zero spaces on the last column~$n$. This yields an injective rectangle submodule of~$\basemod$ (\Cref{lem:dquad-injective}), and therefore a rectangle summand of~$\basemod$ --- a contradiction.

By our claim, $\basemod_{(1,1)}$ maps injectively to the nodes $(i,1)$ for
$i\in \llbracket 1, n-1\rrbracket$. Similary, by restricting~$\basemod$ to
the grid $\llbracket 1, n\rrbracket \times \llbracket 1,
m-1\rrbracket$, we deduce that $\basemod_{(1,1)}$ maps injectively to the
nodes $(1,j)$ for $j\in \llbracket 1, m-1\rrbracket$. Then, by weak
exactness, we have
  \[
  \forall (i,j)\in \llbracket 1,
  n-1\rrbracket \times \llbracket 1, m-1\rrbracket,\,\Ker \mor_{(1,1)}^{(i,j)} = \Ker \mor_{(1,1)}^{(i,1)} + \Ker \mor_{(1,1)}^{(1,j)} =0,
  \]
  so $\basemod_{(1,1)}$ maps injectively to all the nodes of the grid $\llbracket 1, n-1\rrbracket \times \llbracket 1, m-1\rrbracket$.
\end{proof}
\begin{claim}\label{lem:zero-spaces}
The spaces $\basemod_{(1,1)}$ and $\basemod_{(n,m)}$ are zero.
\end{claim}
\begin{proof}
By weak exactness and
\Cref{lem:zero-map}, we have
  \[
  \basemod_{(1,1)} = \Ker \mor_{(1,1)}^{(n,m)} = \Ker \mor_{(1,1)}^{(n,1)} + \Ker \mor_{(1,1)}^{(1,m)}.
  \]
  Assuming for a contradiction that $\basemod_{(1,1)}\neq 0$, we have that at least one of the two terms on the right-hand side of the above equation  must be non-zero -- say $\Ker \mor_{(1,1)}^{(n,1)}\neq 0$. Let $\alpha\neq 0$ be an element in that kernel. By \Cref{lem:injective_mapping}, its images at the nodes of $\llbracket 1, n-1\rrbracket \times \llbracket 1, m-1\rrbracket$ are non-zero. Meanwhile, its images at the nodes of $\{n\}\times \llbracket 1, m\rrbracket$ are zero, by composition. There are two cases:
\begin{itemize}
\item Either $\mor_{(1,1)}^{(1,m)}(\alpha)=0$, in which case the images of
  $\alpha$ at the nodes of $\llbracket 1, n\rrbracket \times \{m\}$ are also
  zero, which implies that the persistence submodule of~$\basemod$ spanned by the
  images of~$\alpha$ is isomorphic to~$\field_{\llbracket 1, n-1\rrbracket
    \times \llbracket 1, m-1\rrbracket}$.
\item Or $\mor_{(1,1)}^{(1,m)}(\alpha)\neq 0$, in which case, for all $i\in \llbracket 1,n-1\rrbracket$, we have
  \[
  \alpha\notin \Ker \mor_{(1,1)}^{(1,m)} \stackrel{\mathrm{\scriptstyle (\Cref{lem:injective_mapping})}}{=} \Ker \mor_{(1,1)}^{(1,m)} + \Ker \mor_{(1,1)}^{(i,1)} = \Ker \mor_{(1,1)}^{(i,m)},
  \]
  which implies that the images of $\alpha$ at the nodes of  $\llbracket 1, n-1\rrbracket \times \{m\}$ are non-zero as well. Hence, the persistence submodule of~$\basemod$ spanned by the images of~$\alpha$ is isomorphic to~$\field_{\llbracket 1, n-1\rrbracket \times \llbracket 1, m\rrbracket}$.
\end{itemize}
In both cases, the persistence submodule of~$\basemod$ spanned by the images of~$\alpha$ is an injective rectangle module (\Cref{lem:dquad-injective}), hence a rectangle summand of~$\basemod$ --- a contradiction.

By applying vector-space duality pointwise to~$\basemod$, we obtain an indecomposable module $\basemod^*$ of the grid $\llbracket 1, n \rrbracket^\op \times \llbracket 1, m \rrbracket^\op$---which is isomorphic to~$\llbracket 1, n \rrbracket \times \llbracket 1, m \rrbracket$ as a poset. This persistence module is still pfd, and still weakly exact as well since the equations of weak exactness are stable under vector-space duality (kernels become images, sums become intersections, and vice-versa). Hence, by the first part of the proof,  $\basemod^*_{(1,1)}=0$, i.e the space at node~$(n,m)$ of $\basemod$ is zero, hence the result.
\end{proof}

\begin{claim}\label{lem:M1m_zero-space}
  The space $\basemod_{(1,m)}$ is zero.
\end{claim}
\begin{proof}
  Assume for a contradiction that $\basemod_{(1,m)}\neq 0$. Call $\omod$ the restriction of $\basemod$ to the grid $\llbracket 1,n-1\rrbracket \times \llbracket 1, m \rrbracket$. By our induction hypothesis, $\omod$ decomposes as a finite (internal) direct sum where each summand is isomorphic to some rectangle module. Since $\basemod_{(1,m)} \neq 0$, at least one of these rectangles contains the node~$(1,m)$. Among such rectangles, take one---say $R'=\llbracket 1, i\rrbracket \times \llbracket j, m\rrbracket$---that has lowest lower-left corner, and call~$N'$ the corresponding summand of $\omod$. Denote by $N''$ the rest of the internal decomposition of~$\omod$, i.e. $N = N'\oplus N''$.

First, we claim that $i=n-1$. Indeed, otherwise we can extend $N'$ to a rectangle persistence submodule~$\bar N'$ of~$\basemod$ by putting zero spaces on the last column~$n$, and~$N''$ to another persistence submodule~$\bar N''$ by putting the internal spaces of~$\basemod$ on the last column, so that $M=\bar N' \oplus \bar N''$---a contradiction.

  Second, we claim that $j\in\llbracket 2, m-1\rrbracket$. Indeed, $j\geq 2$ since by \Cref{lem:zero-spaces} we know that $\basemod_{(1,1)}=0$. Meanwhile, if $j$ were equal to~$m$, then $N'$ would go to zero on the last column of~$\libr 1, n\ribr \times \libr 1, m\ribr$ since $\basemod_{(n,m)}=0$ by \Cref{lem:zero-spaces}, and so we could extend $\omod$ to a rectangle persistence submodule~$\bar N'$ of~$\basemod$ by putting zero spaces on the last column, and~$N''$ to another persistence submodule~$\bar N''$ by putting the internal spaces of~$\basemod$ on the last column, so that $M=\bar N' \oplus \bar N''$---a contradiction.

Consider now the space~$N_{(1,j)}=\basemod_{(1,j)}$, and take a generator~$\alpha$
of the subspace~$N'_{(1,j)}\cong \field$. By
\Cref{lem:zero-spaces} we know that the map
$\mor_{(1,j)}^{(n,m)}$ is zero, so by weak
exactness we have $\alpha = \alpha_h + \alpha_v$ for some
$\alpha_h \in \Ker \mor_{(1,j)}^{(n,j)}$ and $\alpha_v\in \Ker
\mor_{(1,j)}^{(1,m)}$. We claim that~$\alpha_h\notin N''_{(1,j)}$. Indeed, otherwise we would have
\[
\mor_{(1,j)}^{(1,m)}(\alpha) = \mor_{(1,j)}^{(1,m)}(\alpha_h) + \mor_{(1,j)}^{(1,m)}(\alpha_v) = \mor_{(1,j)}^{(1,m)}(\alpha_h) \in \mor_{(1,j)}^{(1,m)}(N''_{(1,j)}) \subseteq N''_{(1,m)},
\]
thus contradicting our assumption that~$N=N'\oplus N''$ with the support of~$N'$ containing~$(1,m)$. Likewise, for any node $t\in R'$ we have $\mor_{(1,j)}^t(\alpha_h)\notin N''_t$, for otherwise we would get a contradiction from
\[
\mor_{(1,j)}^{(t_x,m)}(\alpha) = \mor_{(1,j)}^{(t_x,m)}(\alpha_h) = \mor_{t}^{(t_x,m)}(\mor_{(1,j)}^t(\alpha_h)) \in \mor_{t}^{(t_x,m)}(N''_t) \subseteq N''_{(t_x,m)}.
\]
Thus, the persistence submodule~$N^h$ of~$\omod$ generated by~$\alpha_h$ is isomorphic\footnote{Note that we do not need to check that $\alpha_h$ goes to zero when leaving~$R'$, since by assumption $R'$ reaches row~$m$ and, as we saw earlier, $i=n-1$ so $R'$ reaches column~$n-1$ as well.} to~$N'$ and in direct sum with~$N''$. We can therefore exchange~$N'$ for~$N^h$ in the internal decomposition of~$\omod$. Since $N^h$ is mapped to zero on the last column of~$\libr 1, n\ribr \times \libr 1, m\ribr$, we can extend it to a rectangle persistence submodule~$\bar N^h$ of~$\basemod$ by putting zero spaces on the last column,  meanwhile we can extend~$N''$ to another persistence submodule~$\bar N''$ by putting the internal spaces of~$\basemod$ on the last column, so that $M=\bar N^h \oplus \bar N''$---a contradiction.
\end{proof}

\begin{claim}~\label{lem:M1j_zero-space}
$\basemod_{(1,j)} = 0$ for all $j\in \llbracket 1, m\rrbracket$.
\end{claim}
\begin{proof}
The result is already proven\footnote{It is also proven for $j=1$ by
  \Cref{lem:zero-spaces}, although we do not use this fact in
  the proof.} for $j=m$ by
\Cref{lem:M1m_zero-space}. Let then $j\in \llbracket 1,
m-1\rrbracket$. Call~$\omod$ the restriction of~$\basemod$ to the grid
$\llbracket 1, n\rrbracket \times \llbracket 1, m-1 \rrbracket$. By
our induction hypothesis, $\omod$ decomposes as a finite (internal) direct
sum where each summand is isomorphic to some rectangle
module. Assuming for a contradiction that some summand~$N'$
has a support~$R'$ that intersects the first column, we know from
\Cref{lem:M1m_zero-space} that~$N'$ maps to zero at
node~$(1,m)$. By composition, $N'$ maps to zero as well at the nodes
on the last row~$m$. Therefore, as in the proof of
\Cref{lem:M1m_zero-space}, we can extend~$N'$ to a rectangle
summand of~$\basemod$ by putting zero spaces on row~$m$, thus reaching a
contradiction.
\end{proof}

It follows from \Cref{lem:M1j_zero-space} that~$\basemod$ itself is not supported outside the rectangle~$R=\llbracket 2, n\rrbracket \times \llbracket 1,m\rrbracket$. The induction hypothesis (applied to the restriction of~$\basemod$ to $R$) implies then that~$\basemod$ decomposes as a direct sum of rectangle modules, which raises a contradiction. This concludes the induction step and the proof of \Cref{thm:weak-exact-rec-dec}.

\section{Algorithm for checking rectangle decomposability}
\label{sec:check}
As in Section~\ref{sec:rank_invariant_computation}, let $\filt$ be a simplicial bifiltration with $n$ simplices in total, and let us assume without loss of generality that~$\filt$ is indexed over the grid $\grid=\libr 1, n\ribr \times \libr 1, n\ribr$. We further assume that each arrow $\filt_{(i,j)} \to \filt_{(i+1,j)}$ or $\filt_{(i,j)} \to \filt_{(i,j+1)}$ is either an identity map or a single simplex insertion, and we fix a homology degree~$p$.

Given this input, how fast can we check whether the persistence bimodule~$M$ induced in $p$-th homology decomposes into rectangle summands? An obvious solution is to first decompose~$M$ from the data of~$\filt$, then to check the summands one by one. As explained in Section~\ref{sec:rank_invariant_computation},  the currently best known algorithm for decomposition runs in time~$O(n^{2\omega + 1})$, where $2\leq \omega < 2.373$ is the exponent for matrix multiplication~\cite{dey2019generalized}. The advantage of the algebraic weak exactness condition from Section~\ref{sec:local_condition} is that it can be checked locally, which reduces the total running time to $O(n^{2+\omega})$. Below we sketch the algorithm:
\begin{enumerate}
\item Compute the rank invariant $\rank: \libr 1, n\ribr^2 \times \libr 1, n\ribr^2 \to\N$ of~$M$.
\item Compute invariants for kernels and images, denoted by $\kappa: \libr 1, n\ribr^2 \times \libr 1, n\ribr^2 \to\N$ and $\iota: \libr 1, n\ribr^2 \times \libr 1, n\ribr^2 \to\N$ respectively, which return the dimensions of $\Ker \mor_s^{(s_x,t_y)} + \Ker \mor_s^{(t_x,s_y)}$ and of $\Ima \mor_{(s_x,t_y)}^t \cap \Ima \mor_{(t_x,s_y)}^t$ respectively at indices $s\leq t$, and zero elsewhere. 
\item For each pair of indices $s\leq t$, check whether $\rank(s,t) = \iota(s,t)$ and $\rank(s,s)-\rank(s,t) = \kappa(s,t)$. If any such equality fails, then answer that $M$ is not rectangle-decomposable. Otherwise, answer that $M$ is rectangle-decomposable.
\end{enumerate}
We now provide further implementation details and analyze the algorithm on the fly:

Step~1 has already been detailed in Section~\ref{sec:rank_invariant_computation} and runs in $O(n^4)$~time.

Step~3 obviously runs in $O(n^4)$ time, and its correctness comes from the commutativity of the square in~\eqref{eq:weak-exactness}: indeed, commutativity  implies that $\Ima \mor_s^t\subseteq \Ima \mor_{(s_x,t_y)}^t \cap \Ima \mor_{(t_x,s_y)}^t$ and $\Ker \mor_s^{(s_x,t_y)} + \Ker \mor_s^{(t_x,s_y)}\subseteq \Ker \mor_s^t$, so checking weak exactness for this square amounts to checking equality between the dimensions of the various spaces involved, hence the equations.

For Step~2, we first compute, for each $t=(j,l)\in\grid$, the barcode of the zigzag module\footnote{A zigzag module is a persistence module indexed over a poset of the form $\xymatrix{\bullet\ar@{<->}[r] & \bullet\ar@{<->}[r] & \cdots \ar@{<->}[r] & \bullet}$, where double-headed arrows mean that the actual arrows can be oriented either forward or backward. Such modules always decompose into direct sums of interval modules~\cite{botnan2015interval,carlsson2010zigzag}.} induced in homology by the following zigzag of simplicial complexes:
\begin{align}\label{eq:ZZ_ima}
    \xymatrix{\filt_{(1,l)}\ar[r]&\cdots\ar[r]&\filt_{(j-1,l)}\ar[r]&\filt_t&\ar[l]\filt_{(j,l-1)}&\ar[l]\cdots&\ar[l]\filt_{(j,1)}}.
  \end{align}
We then do the same with the following zigzag, for  each $s=(i,k)\in\grid$:
\begin{align}\label{eq:ZZ_ker}
    \xymatrix{\filt_{(i,n)}&\ar[l]\cdots&\ar[l]\filt_{(i,k+1)}&\ar[l]\filt_s\ar[r]&\filt_{(i+1,k)}\ar[r]&\cdots\ar[r]&\filt_{(n,k)}}.
  \end{align}
Then, for each indices~$(i,k)=s\leq t=(j,l)$, by restriction,  the dimension of $\Ima \mor_{(i,l)}^t \cap \Ima \mor_{(j,k)}^t$ is given by the number of intervals in the barcode of~\eqref{eq:ZZ_ima} that span the subzigzag~$\xymatrix{\filt_{(i,l)}\ar[r]&\filt_t&\ar[l]\filt_{(j,k)}}$, while (dually) the dimension of $\Ker \mor_s^{(i,l)} + \Ker \mor_s^{(j,k)}$ is given by $\rank(s,s)$ minus the number of intervals in the barcode of~\eqref{eq:ZZ_ker} that span the subzigzag~$\xymatrix{\filt_{(i,l)}&\ar[l]\filt_s\ar[r]&\filt_{(j,k)}}$ (the proof of these simple facts is given in Appendix~\ref{sec:proof_simple_facts}). Regarding the running time: since the zigzags~\eqref{eq:ZZ_ima}-\eqref{eq:ZZ_ker} involve $O(n)$ simplex insertions and deletions each, their barcode computation takes $O(n^\omega)$ using the algorithm based on fast matrix multiplication~\cite{milosavljevic2011zigzag}. Then, each barcode having $O(n)$ intervals, the computation of the dimensions and their storage in tables of integers representing the invariants~$\kappa$ and~$\iota$ takes $O(n)$. This is true for each choice of indices~$s\leq t$, hence a total running time in~$O(n^{2+\omega}+ n^3)= O(n^{2+\omega})$.
As a consequence,
\begin{theorem}\label{thm:check_weak_exactness}
  Checking the rectangle-decomposability of the bimodule induced in homology by a bifiltration with $n$ simplices in total can be done in $O(n^{2+\omega})$~time.
\end{theorem}
%

\section{Local characterizations for larger classes of indecomposables}
\label{sec:counter-examples}
\Cref{main:rec-dec} ensures that rectangle-decomposability of a given persistence module can be checked locally by considering restrictions to commutative squares. A natural next question to consider is then: to what extent can interval-decomposability be locally determined when allowing for intervals of more general shape than rectangles? In this section we provide two negative results: We show that interval-decomposability itself cannot be characterized locally, even when testing on arbitrary strict subgrids. Then we show that decomposability into a class of interval modules strictly containing rectangle modules cannot be locally determined by means of restrictions to squares. 

\subsection{Characterizing interval-decomposability}
\label{sec:simple-cases}
\subsubsection{Testing on totally ordered subsets.}
Denote by~$\totallysubpos[\poset]$ the set of all totally ordered subsets of~$\poset$. It has been shown in \cite{Crawley-Boevey2012,botnan2018decomposition} that any pfd persistence module indexed over a totally ordered set is interval-decomposable. Therefore, if~$M$ is a pfd persistence bimodule indexed over~$\poset$, the restriction~$M_{|Q}$ is interval-decomposable for any~$Q\in\totallysubpos[\poset]$. Hence, any indecomposable module over~$\poset$ that is not of pointwise dimension~$0$ or~$1$ (such as the one defined in \Cref{ex:3-2-indecomposable}) is a counter-example to the existence of a local characterization of interval-decomposability over totally ordered subsets.

\subsubsection{Testing on squares.} 
Recall that the restriction of a pfd persistence module over~$\poset$ to any commutative square is interval-decomposable (see e.g. Figure 13 in \cite{escolar2016persistence}). Therefore, any indecomposable module over~$\poset$ that is not of pointwise dimension~$0$ or~$1$ (see again \Cref{ex:3-2-indecomposable}) is a counter-example to the existence of a local characterization of interval-decomposability over squares.

\subsubsection{Testing on finite grids of bounded size.}

Our analysis proceeds in two steps: first we consider the special case where~$\poset$ is the finite grid~$\libr 1, n+1\ribr^2$, then we extend the analysis to the case of general finite product posets. The intuition behind our constructions is given in the following section.

\paragraph{A minimal indecomposable.}
\label{sec:indecomposable}
Let~$n\geq 2$ be an integer, and consider the poset~$\dart{n}$  given by the following Hasse diagram:
\[
\xymatrix@=10pt{
  1 \ar[rrrr] &&&& n+2 \\
  & 2\ar[urrr] \\
  && \ddots \\
  &&& n \ar[uuur] \\
  &&&& n+1 \ar[uuuu]
}
\]
Denote by~$\iota_i$ the inclusion of the~$i$-th axis~$\field\hookrightarrow\field^n$, and by~$\delta_n$ the injection into the diagonal~$t\in\field\mapsto (t,\dots,t)\in\field^n$. Let ~$\indec{n}$ denote the persistence module over~$\dart{n}$ --- which
can be easily be seen as a subposet of~$\R^2$ --- given by the following diagram:
\[
\xymatrix@=10pt{
  \field \ar^-{\iota_1}[rrrr] &&&& \field^n \\
  & \field \ar^-{\iota_2}[urrr] \\
  && \ddots \\
  &&& \field \ar^-{\iota_n}[uuur] \\
  &&&& \field \ar_-{\delta_n}[uuuu]
}
\]
\begin{lemma}\label{lem:indecomposable}
    The persistence module~$\indec{n}$ satisfies:
    \begin{enumerate}
    \item $\indec{n}$ is indecomposable with local endomorphism ring;
    \item for any~$i\in\libr 1,n+1\ribr$, the restriction~$\indec{n}_{|\dart{n}\setminus\{i\}}$ decomposes as follows:
        \begin{equation*}
        \indec{n}_{|\dart{n}\setminus\{i\}} \cong \bigoplus_{j\in\libr 1,n+1\ribr\setminus\{i\}} \indimod[\{j,n+2\}],
        \end{equation*}
        where~$\indimod[\{j,n+2\}]$ is the indicator module of the set~$\{j, n+2\}$.
    \end{enumerate}
\end{lemma}

\begin{proof}
    It is straightforward to check that~$\indec{n}$ has endomorphism ring isomorphic to~$\field$, which is local. Therefore, $\indec{n}$ is indecomposable. Now, if~$i = n+1$ then the decomposition of~$\indec{n}_{|\dart{n}\setminus\{i\}}$ is obvious, while if~$i\in\libr 1, n\ribr$ then a simple change of basis in the space~$\indec{n}_{n+2}$ yields an isomorphism between~$\indec{n}_{|\dart{n}\setminus\{i\}}$ and~$\indec{n}_{|\dart{n}\setminus\{n+1\}}$ via the identification~$\dart{n}\setminus\{i\} \simeq \dart{n}\setminus\{n+1\}$, which brings us back to the case where~$i = n+1$.
\end{proof}

\paragraph{Negative result for the poset~$\libr 1, n+1\ribr^2$}
\label{sec:special-case_sqgrid}

Given~$n\geq 2$, define the following persistence module over~$\libr 1,n+1\ribr^2$, where dotted lines stand for zero maps or chains of zero maps, unspecified solid lines stand for identity maps, and dashed lines stand for chains of identity maps: 
\begin{equation}
    \label{eqn:def-indec-grid}
    \indecgrid{n} :=
    \begin{tikzcd}[arrows=-stealth]
        \field \arrow[r, "\iota_1"]            & \field^n \arrow[r]                               & \field^n \arrow[r, dashed]           & \field^n \arrow[r]                               & \field^n                     \\
        0 \arrow[u, dotted] \arrow[r, dotted]  & \field \arrow[u, "\iota_2"] \arrow[r, "\iota_2"] & \field^n \arrow[u] \arrow[r, dashed] & \field^n \arrow[r] \arrow[u]                     & \field^n \arrow[u]           \\
                                               &                                                  & \ddots                               & \field^n \arrow[u, dashed] \arrow[r]             & \field^n \arrow[u, dashed]   \\
        0 \arrow[r, dotted] \arrow[uu, dotted] & 0 \arrow[uu, dotted] \arrow[rr, dotted]          &                                      & \field \arrow[r, "\iota_n"] \arrow[u, "\iota_n"] & \field^n \arrow[u]           \\
        0 \arrow[u, dotted] \arrow[r, dotted]  & 0 \arrow[u, dotted] \arrow[rr, dotted]           &                                      & 0 \arrow[u, dotted] \arrow[r, dotted]            & \field \arrow[u, "\delta_n"]
        \end{tikzcd}
\end{equation}
In more abstract terms, defining the monomorphism of posets:
\begin{equation}
    \phi :
    \begin{array}{ccl}
        \dart{n} &\hookrightarrow & \libr 1, n+1\ribr ^2 \\[1ex]
        i & \mapsto &
        \begin{cases}
            (i, n+2-i) &(i \ne n+2)\\
            (n+1,n+1) &(i = n+2)
        \end{cases}
    \end{array}
    ,
\end{equation}
we have:
\begin{equation}
    \label{eq:charac-indecgrid-Ran}
    \indecgrid{n} \simeq \Ran_\phi \indec{n}.
\end{equation}
This can be easily seen from the description of right Kan extensions  of persistence modules as "floor" modules~\cite[Sec.~2.5]{Botnan2016}. More precisely, we have for all~$t\in\libr 1,n+1\ribr^2$:
\begin{equation*}\label{eqn:characterization-ran}
    \indecgrid{n}_t = \lim \indec{n}_{|\phi_{\geq t}},
\end{equation*}
where~$\phi_{\geq t}$ denotes the upset~$\{u\in\dart{n} \,|\, \phi(u) \geq t\}$. Internal morphisms for~$s\leq t$ in~$\libr 1,n+1\ribr^2$ are given by the universality of limits. 
%
%
\begin{proposition}
    \label{prop:counter-example-grid}
    For $n\geq 2$, the persistence module~$\indecgrid{n}$ satisfies:
    \begin{enumerate}
        \item\label{itm:indecomposable}~$\indecgrid{n}$ is not interval-decomposable;
        \item\label{itm:minimal} for any strict subgrid~$X'\times Y'\subsetneq \libr 1,n+1\ribr^2$, the restriction~$\indecgrid{n}_{|X'\times Y'}$ is interval-decomposable.
    \end{enumerate}
\end{proposition}

\begin{proof}
    The monomorphism $\phi$ being fully faithful,  \Cref{lem:endoring-Kan} implies that the endomorphism ring of~$\indecgrid{n}$ is isomorphic to that of~$\indec{n}$, which is local by \Cref{lem:indecomposable}. Hence, $\indecgrid{n}$ is indecomposable, and since it is not of pointwise dimension $0$ or $1$, it is not an interval module. This proves item~\ref{itm:indecomposable}.
For~\ref{itm:minimal}, since any strict subgrid~$X'\times Y'$ of~$\libr 1,n+1\ribr ^2$ misses at least one row or one column of~$\libr 1,n+1\ribr ^2$, we will merely show that the restriction of~$\indecgrid{n}$ to~$\libr 1,n+1\ribr ^2 \setminus C$, where~$C$ denotes an arbitrary column of~$\libr 1,n+1\ribr ^2$, is interval-decomposable. Indeed, the result for~$\libr 1,n+1\ribr ^2 \setminus R$ where~$R$ is a row of~$\libr 1,n+1\ribr ^2$ is obtained analogously, and then the result for the restriction of~$\indecgrid{n}$ to any strict subgrid~$X'\times Y'$ follows by restriction.

    A column $C$ of $\libr 1,n+1\ribr ^2$ contains exactly one point of the form $(i,n+2-i)$ for an $i\in\libr 1,n+1\ribr$. We denote by $C_i$ the column containing the point $(i, n+2-i)$. Hence, the corestriction of $\phi_{|\dart{n}\setminus\{i\}}$ to $\libr 1,n+1\ribr ^2\setminus C_i$ is well-defined, denote it by $\phi_i : \dart{n}\setminus\{i\} \hookrightarrow \libr 1,n+1\ribr ^2\setminus C_i$, and we can easily see that:
    \begin{equation}
        \label{eq:charac-indecgrid-Ran-restriction}
        \indecgrid{n}_{|\libr 1,n+1\ribr ^2\setminus C_i} \simeq \Ran_{\phi_i} \left( \indec{n}_{|\dart{n}\setminus\{i\}} \right).
    \end{equation}
    Moreover, for any $j\in\libr 1,n+1\ribr\setminus\{i\}$, the module $\Ran_{\phi_i} \field_{\{j,n+2\}}$ is clearly an interval module, and in particular the finite direct sum $\bigoplus_{j\in\libr 1,n+1\ribr\setminus\{i\}} \Ran_{\phi_i}\field_{\{j,n+2\}}$ is pointwise-finite dimensional. For instance, $\Ran_{\phi_{n+2}} \field_{\{1,n+2\}}$ is isomorphic to:
    \begin{equation}
        \begin{tikzcd}[arrows=-stealth]
            \field \arrow[r]               & \field \arrow[r]                               & \field \arrow[r, dashed]             & \field                         \\
            0 \arrow[u, dotted] \arrow[r, dotted]     & 0 \arrow[u, dotted] \arrow[r, dotted]          & \field \arrow[u] \arrow[r, dashed]   & \field \arrow[u]            \\
                                                      &                                                & \ddots                               & \field \arrow[u, dashed]    \\
            0 \arrow[r, dotted] \arrow[uu, dotted]    & 0 \arrow[uu, dotted] \arrow[rr, dotted]        &                                      & 0 \arrow[u, dotted]   \\
            0 \arrow[u, dotted] \arrow[r, dotted]     & 0 \arrow[u, dotted] \arrow[rr, dotted]         &                                      & 0 \arrow[u, dotted]          
        \end{tikzcd}
    \end{equation}
    Therefore, using \Cref{lem:indecomposable} and the fact that pfd right Kan extensions commute with direct-sums of pfd modules~\cite[Rk.~2.16]{Botnan2016}, we get an interval-decomposition:
    \begin{equation*}
        \Ran_{\phi_i} \left( \indec{n}_{|\dart{n}\setminus\{i\}} \right) \simeq \bigoplus_{j\in\libr 1,n+1\ribr\setminus\{i\}} \Ran_{\phi_i}\field_{\{j,n+2\}}, 
    \end{equation*}
    hence the result by \eqref{eq:charac-indecgrid-Ran-restriction}.
\end{proof}

\Cref{prop:counter-example-grid} immediately implies a similar result for more general finite grids.

\begin{corollary}
    \label{cor:negative-int-dec-subgrids}
    If $X,Y$ are arbitrary finite subsets of $\R$ with $|X|\geq 3$, $|Y|\geq 3$, then there exists a pfd persistence module $M$ indexed over $\poset$ such that:
    \begin{enumerate}
        \item $M$ is not interval-decomposable; 
        \item for any subgrid~$X'\times Y'\subsetneq \poset$ with $|X'| < \min(|X|,|Y|)$ or $|Y'| < \min(|X|,|Y|)$, the restriction~$M_{|X'\times Y'}$ is interval-decomposable.
    \end{enumerate}
\end{corollary}

\begin{proof}
    Let~$m = \min(|X|,|Y|)$. Since $m \geq 3$, we can define a persistence module $M$ on $\poset$ by copy-pasting the spaces and maps of $\indecgrid{m-1}$ to the most bottom-left subgrid of size $m \times m$ of $\poset$, and by setting all other spaces and maps to zero. The result follows then from \Cref{prop:counter-example-grid}.
\end{proof}

\subsection{Local characterizations for other classes of interval-decomposable modules}
\label{sec:adding-hook} 
Since rectangle-decomposability can be characterized locally on squares, while general interval-decomposability cannot,  it is natural to ask what happens with classes of indecomposables standing in-between the rectangle modules and the interval modules. \Cref{thm:hooks} below shows that certain such classes cannot be characterized locally on squares either. In the following, we use the notations of~\eqref{eq:def-int-square} for the intervals of a square---note that only two of these intervals are not rectangles, namely $\{a,b,c\}$ (called \emph{bottom hook}) and $\{b,c,d\}$ (called \emph{top hook}). 

\begin{theorem}\label{thm:hooks}
    Let $X,Y$ be arbitrary finite subsets of $\R$ such that $|X| \geq 2$ and $|Y| \geq 2$, and that $(|X|, |Y|) \neq (2, 2)$. Then, there exists a pfd persistence module $M$ indexed over $\poset$ such that:
    \begin{enumerate}
        \item $M$ is not interval-decomposable; 
        \item for any square~$Q$ of $\poset$, the restriction~$M_{|Q}$ is interval-decomposable and its barcode is included in $\mathcal{R}\cup\{\{b,c,d\}\}$.
    \end{enumerate}
    The same result holds for $\mathcal{R}\cup\{\{a,b,c\}\}$.
\end{theorem}

\begin{proof}
    Recall that the persistence module to the right of \Cref{ex:3-2-indecomposable} is indecomposable with local endomorphism ring. Since it is not of pointwise dimension 0 or 1, it is not interval-decomposable. Furthermore, direct inspection reveals that its restrictions to squares are all interval-decomposable, with their barcodes included in $\mathcal{R}\cup\{\{b,c,d\}\}$---in fact in $\mathcal{R}$ except for the restriction to the outermost square whose barcode is made of one copy of $\{b,c,d\}$. Thus, when $|X| \geq  3$ and $|Y| \geq 2$, we can copy-paste the spaces and maps of this bimodule to the most bottom-left subgrid of size $3 \times 2$ of $\poset$, and set all other spaces and maps to zero, to get our module~$M$. When $|X| \geq 2$ and $|Y| \geq 3$, the result is proven similarly using the following vertical analogue of the previous bimodule:
    \begin{equation*}
        \begin{tikzcd}[arrows=-stealth, ampersand replacement=\&]
            \field \arrow[r, "1"]                              \& \field                               \\
            \field \arrow[u, "1"] \arrow[r, "\mymatrix{1\\0}"] \& \field^2 \arrow[u, "\mymatrix{1&0}"'] \\
            0 \arrow[u, "0"] \arrow[r, "0"]                    \& \field \arrow[u, "\mymatrix{1\\1}"'] 
        \end{tikzcd}.
    \end{equation*}
    Taking the duals%
    \footnote{See the end of the proof of \Cref{lem:zero-spaces} for more details on duality.}
    of the previously constructed persistence modules yields the result for~$\mathcal{R}\cup\{\{a,b,c\}\}$.
\end{proof}

\section{New proof of the Pyramid basis theorem}
\label{sec:example}
In \cite{carlsson2009zigzag} the authors show that a large pyramidal diagram can be associated to a sufficiently tame real valued function $f : X \to \R$. We briefly recall their construction.  Under the assumption that the function is of \emph{Morse type}, there exists a finite set of \emph{critical values} $a_1 < a_2 < \ldots < a_n$, and we may choose real values $s_i$ satisfying
\begin{equation}
  -\infty < s_0 < a_1 < s_1 < \dots < a_n < s_n < +\infty.
\end{equation}
Here the idea is that the preimage of $[s_i, s_{i+1}]$ deformation retracts onto the fiber over $a_{i+1}$, and that the fiber is constant (up to homotopy) between critical values. This gives a way of studying how the topology of the fibers connect across scales. 

Denoting $X_i^j = f^{-1}[s_i,s_j]$ and $ {}_i^jX = X_0^i \cup X_j^n$, obvious inclusions yield a commutative diagram, such as the following one for $n=2$:

%
%
%

\medskip
\begin{tikzpicture}[baseline= (a).base]
\node[scale=.9] (a) at (0,0){
  \begin{tikzcd}
    &                                     & {(X_0^2,X_0^2)} \arrow[r]           & {(X_0^2, {}_2^2X)}                     &                                        &                           \\
    & {(X_0^1,X_0^1)} \arrow[r]           & {(X_0^2,X_0^1)} \arrow[r] \arrow[u] & {(X_0^2, {}_1^2X)} \arrow[u] \arrow[r] & {(X_0^2, {}_1^1X)}                     &                           \\
{(X_0^0,X_0^0)} \arrow[r] & {(X_0^1,X_0^0)} \arrow[u] \arrow[r] & {(X_0^2,X_0^0)} \arrow[u] \arrow[r] & {(X_0^2, {}_0^2X)} \arrow[u] \arrow[r] & {(X_0^2, {}_0^1X)} \arrow[u] \arrow[r] & {(X_0^2, {}_0^0X)}        \\
X_0^0 \arrow[r] \arrow[u] & X_0^1 \arrow[r] \arrow[u]           & X_0^2 \arrow[r] \arrow[u]           & {(X_0^2,X_2^2)} \arrow[r] \arrow[u]    & {(X_0^2,X_1^2)} \arrow[r] \arrow[u]    & {(X_0^2,X_0^2)} \arrow[u] \\
    & X_1^1 \arrow[u] \arrow[r]           & X_1^2 \arrow[u] \arrow[r]           & {(X_1^2,X_2^2)} \arrow[u] \arrow[r]    & {(X_1^2,X_1^2)} \arrow[u]              &                           \\
    &                                     & X_2^2 \arrow[u] \arrow[r]           & {(X_2^2,X_2^2)} \arrow[u]              &                                        &                          
\end{tikzcd}};
\end{tikzpicture}
\medskip

Building on the work of \cite{carlsson2009zigzag}, it is shown in \cite{bendich2013homology}  that the above diagram, upon application of homology, decomposes into a direct sum of interval modules, where each interval is the intersection of a rectangle in $\Z^2$ with the pyramid above. This result is referred to as the \emph{pyramid basis theorem}. We now give a new proof of this fact using \cref{thm:weak-exact-rec-dec}. More precisely, we show the following:
\begin{theorem}[Pyramid basis theorem]
The homology pyramid as constructed in \cite{carlsson2009zigzag} is interval-decomposable, where the intervals are restrictions of rectangles in $\Z^2$ to the pyramid. 
\end{theorem}
To simplify the notation we prove the case for $n=2$ and it will be evident that the argument generalizes. 
First, extend the homology diagram to a bimodule on a finite grid as follows:

\medskip
\begin{tikzpicture}[baseline= (a).base]
\node[scale=.75] (a) at (0,0){
  \begin{tikzcd}
   0\ar[r]&   0       \ar[r]                           & H_p{(X_0^2,X_0^2)} \arrow[r]           & H_p{(X_0^2, {}_2^2X)}    \ar[r]                 &            \text{PO}_2        \ar[r]                   &             \text{PO}_3            \\
0\ar[r] \ar[u]  & H_p{(X_0^1,X_0^1)} \arrow[r]  \ar[u]     & H_p{(X_0^2,X_0^1)} \arrow[r] \arrow[u] & H_p{(X_0^2, {}_1^2X)} \arrow[u] \arrow[r] & H_p{(X_0^2, {}_1^1X)}                     \ar[r]\ar[u]&                   \text{PO}_1  \ar[u]     \\
\ar[u] H_p{(X_0^0,X_0^0)} \arrow[r] & H_p{(X_0^1,X_0^0)} \arrow[u] \arrow[r] & H_p{(X_0^2,X_0^0)} \arrow[u] \arrow[r] & H_p{(X_0^2, {}_0^2X)} \arrow[u] \arrow[r] & H_p{(X_0^2, {}_0^1X)} \arrow[u] \arrow[r] & H_p{(X_0^2, {}_0^0X)}  \ar[u]      \\
H_p(X_0^0) \arrow[r] \arrow[u] & H_p(X_0^1 )\arrow[r] \arrow[u]           & H_p(X_0^2) \arrow[r] \arrow[u]           & H_p{(X_0^2,X_2^2)} \arrow[r] \arrow[u]    & H_p{(X_0^2,X_1^2)} \arrow[r] \arrow[u]    & H_p{(X_0^2,X_0^2)} \arrow[u] \\
   \text{PB}_1\ar[u]\ar[r] & H_p(X_1^1) \arrow[u] \arrow[r]           & H_p(X_1^2) \arrow[u] \arrow[r]           & H_p{(X_1^2,X_2^2)} \arrow[u] \arrow[r]    & H_p{(X_1^2,X_1^2)} \arrow[u]         \ar[r]     &                       \ar[u]   0 \\
    \text{PB}_3\ar[r]\ar[u]  &                \text{PB}_2 \ar[r]\ar[u]                     & H_p(X_2^2) \arrow[u] \arrow[r]           & H_p{(X_2^2,X_2^2)} \arrow[u]   \ar[r]           &                  0  \ar[r]\ar[u]                    &               \ar[u]    0       
\end{tikzcd}};
\end{tikzpicture}
\medskip

Here $\text{PB}_i$ denotes pullback and $\text{PO}_i$ denotes pushout.  Inductively these are defined (up to canonical isomorphism) by 
\begin{equation*} 
  \begin{split}
    \text{PB}_1 &= \ker \left(H_p(X_0^0)\oplus H_p(X_1^1) \to H_p(X_0^1)\right)\\
    \text{PB}_2 &= \ker\left(H_p(X_1^1)\oplus H_p(X_2^2)\to H_p(X_1^2)\right)\\
    \text{PB}_3 &= \ker\left(\text{PB}_1\oplus \text{PB}_2\to H_p(X_1^1) \right).
  \end{split}
\end{equation*}
and dually for the pushouts, with kernels replaced by cokernels.  By \cref{thm:weak-exact-rec-dec} it suffices to show that the extended diagram is weakly exact. The fact that any square with vertices on the original ''pyramid'' is strongly exact (i.e. the sequence~\eqref{eq:exact-seq} induced by such a square is exact) follows from the exactness of the relative Mayer--Vietoris sequence. Morever, as remarked in \cite[Section~5.1]{botnan2018decomposition}, the extension of the pyramid to pullbacks and pushouts preserves strong exactness (and thus weak exactness). It remains to consider squares with a 0 vector space as either its top-left or bottom-right corner. The fact that such squares are weakly exact is an easy consequence of commutativity. We conclude that the bimodule shown above is weakly exact and therefore rectangle-decomposable. The restrictions of the rectangle summands to the original homology pyramid give the intervals in the \emph{pyramid basis theorem}. 
\bibliography{mybib}

\appendix

\section{Proof of the simple facts from Section~\ref{sec:check}} 
\label{sec:proof_simple_facts}
\begin{lemma}\label{lem:simple_facts}
  Consider the following commutative square (left) and pfd persistence bimodule indexed over it (right):
  \[
  \xymatrix{
    \bullet_c \ar[r] & \bullet_d &&& C \ar^-\delta[r] & D\\
    \bullet_a \ar[u]\ar[r] & \bullet_b \ar[u] &&& A \ar^-\beta[u] \ar^-{\alpha}[r] & B \ar_-{\gamma}[u]
  }
  \]
  Then:
  \begin{align*}
    \dim (\Ima \gamma \cap \Ima \delta) \quad = & \quad \#\left\{ \mbox{intervals of type}\ \xymatrix{\bullet_c\ar@{-}[rr] & \bullet^d & \bullet_b}\ \mbox{in the barcode}\ \right.\\
    & \quad\quad\quad\quad\quad\quad\quad\quad\quad\quad\quad \left. \mbox{of the zigzag}\ \xymatrix{C \ar^-\delta[r] & D & \ar_-{\gamma}[l] B}\right\}; \\[0.5em]
    \dim (\Ker \alpha + \Ker \beta) \quad = & \quad \dim (A) - \#\left\{ \mbox{intervals of type}\ \xymatrix{\bullet_c\ar@{-}[rr] & \bullet_a & \bullet_b}\ \mbox{in the}\ \right.\\
    & \quad\quad\quad\quad\quad\quad\quad\quad \left. \mbox{barcode of the zigzag}\ \xymatrix{C &  \ar_-\beta[l] A \ar^-{\alpha}[r] & B}\right\}. \\
  \end{align*}
  \end{lemma}
\begin{proof}
  We only prove the first equality, as the second one follows by duality. Take an interval decomposition of the zigzag $\xymatrix{C \ar^-\delta[r] & D & \ar_-{\gamma}[l] B}$, and pick a basis~$(\xi_1, \cdots, \xi_l)$ of~$D$ that is compatible with this decomposition. This means that each basis element~$\xi_i$ lies in the span of a unique interval summand of the zigzag at~$D$. Then, by restriction we have:
\begin{align*}
  \xi_i\in \Ima \gamma \quad \Longleftrightarrow & \quad \xi_i \in \Span \left(\mbox{summands of type}\ \xymatrix{\bullet_d\ar@{-}[r] & \bullet_b}\right)\\[0.5em]
  & \quad\quad + \Span \left(\mbox{summands of type}\ \xymatrix{\bullet_c\ar@{-}[rr] & \bullet^d & \bullet_b}\right); \\[1em]
  \xi_i\in \Ima \delta \quad \Longleftrightarrow & \quad \xi_i \in \Span \left(\mbox{summands of type}\ \xymatrix{\bullet_c\ar@{-}[r] & \bullet_d}\right)\\[0.5em]
  & \quad\quad + \Span \left(\mbox{summands of type}\ \xymatrix{\bullet_c\ar@{-}[rr] & \bullet^d & \bullet_b}\right).
\end{align*}
The spans of distinct summands being in direct sum in~$D$, we deduce that
\[
\xi_i \in \Ima \gamma \cap \Ima \delta \quad \Longleftrightarrow \quad \xi_i \in \Span \left(\mbox{summands of type}\ \xymatrix{\bullet_c\ar@{-}[rr] & \bullet^d & \bullet_b}\right).
\]
Hence the result.  
\end{proof}

\section{Endomorphism rings of Kan extensions}
\label{sec:lem-kan}
\begin{lemma}
    \label{lem:endoring-Kan}
    Let $P$ and $Q$ be two subposets of $\R^d$, $M$ be a pfd persistence module over~$P$ and $\phi : P \hookrightarrow Q$ be a fully faithful monomorphism of posets. Then, the endomorphism ring of $\Lan_\phi M$ (resp. of $\Ran_\phi M$) is isomorphic to that of~$M$.
  \end{lemma}
  \begin{proof}
    We prove the result for left Kan extensions, the case of right Kan extensions being similar. Since $\phi$ is fully faithful, we have by \cite[Cor.~X.3.3]{MacLane1998} that:
    \begin{equation}
      \label{eq:Lan-restric}
      \phi^*\left( \Lan_\phi M \right) \simeq M,
    \end{equation}
    where $\phi^*$ denotes the functor ``pre-composition by~$\phi$'' going from the category of pfd persistence modules over~$Q$ to the category of pfd persistence modules over~$P$. Moreover, the universality property of Kan extensions 
    gives a natural isomorphism:
    \begin{equation}
      \label{eq:adjunction-Lan}
      \Hom(\Lan_\phi M,-) \simeq \Hom(M, \phi^*(-)).
    \end{equation}
    Combining these two equations, we get:
    \begin{align*}
      \End(M) & \hspace{1.25em} = \hspace{1.2em} \Hom(M,M) \\
      &\overset{\text{Eq. \eqref{eq:Lan-restric}}}{\simeq} \Hom(M,\phi^*(\Lan_\phi M)) \\
      &\overset{\text{Eq. \eqref{eq:adjunction-Lan}}}{\simeq} \Hom(\Lan_\phi M, \Lan_\phi M) \\
      & \hspace{1.25em} = \hspace{1.2em} \End(\Lan_\phi M).
    \end{align*}
  \end{proof}

\end{document}